\documentclass[reqno,12pt]{amsart}

\address{Department of Mathematics\\ California Institute of Technology\\ 1200 E California Blvd\\
MC253-37\\
\\Pasadina\\ CA 91125\\ USA}
\email{agholamp@its.caltech.edu}
\address{Department of Mathematics\\ University of Utah\\ 155 S 1400 E JWB 233\\Salt Lake city\\ UT 84112\\ USA}
\email{jiangyf@math.utah.edu}

\usepackage{amsmath,amsthm,amsfonts,amscd}
\usepackage{times}
\usepackage{verbatim}
\usepackage{afterpage}
\usepackage{xy}
\usepackage{graphicx}
\usepackage{enumerate}
\usepackage{diagrams}
\usepackage{amsmath,amsfonts,amssymb}
\usepackage{bbm,dsfont}
\usepackage{mathrsfs}
\usepackage{ifthen}
\usepackage{rotating}
\usepackage{latexsym}
\usepackage[centering,text={14.5cm,19.5cm},
		marginparwidth=20mm]{geometry}

\newtheorem{prop}{Proposition}[section]
\newtheorem{lem}[prop]{Lemma}

\newtheorem{cor}[prop]{Corollary}
\newtheorem{corst}[prop]{Corollary$\,\mathbf{^\star}$}
\newtheorem{thm}[prop]{Theorem}
\newtheorem{thmst}[prop]{Theorem$\,\mathbf{^\star}$}
\newtheorem{example}[prop]{Example}
\newtheorem{rmk}[prop]{Remark}

\newtheorem{conjecture}[prop]{Conjecture}

\newtheorem{defn}[prop]{Definition}

\newcommand{\rep}{\text{Rep}}
\newcommand{\ah}{\widehat{\alpha}}
\newcommand{\fl}{\operatorname{Flag}}
\newcommand{\f}{\operatorname{Fl}}
\newcommand{\R}{\widehat{R}}
\newcommand{\res}{\text{Res}}
\newcommand{\ms}{\MM_{\zeta}(\mathbf{v})}
\newcommand{\msc}{\MM^{\cc^*}_{\zeta}(\mathbf{v})}
\newcommand{\msp}{\MM_{\zeta_\alpha^+}(\mathbf{v'})}
\newcommand{\msn}{\MM_{\zeta_\alpha^-}(\mathbf{v'})}
\newcommand{\mspn}{\MM_{\zeta_\alpha^{\pm}}(\mathbf{v'})}

\newcommand{\mscpn}{\MM^{\cc^*}_{\zeta_\alpha^{\pm}}(\mathbf{v'})}
\newcommand{\msa}{\MM_{\zeta_\alpha}(\mathbf{v})}
\newcommand{\msca}{\MM_{\zeta_\alpha}^{\cc^*}(\mathbf{v})}

\newcommand{\zetat}{\tilde{\zeta}}
\newcommand{\A}{\tilde{A}}
\newcommand{\Qt}{\tilde{\Q}}
\newcommand{\Q}{\mathcal{Q}}
\newcommand{\D}{\mathcal{D}}
\newcommand{\V}{\tilde{V}}
\newcommand{\hilb}{\text{-Hilb}}
\newcommand{\Irr}{\operatorname{Irr}}
\newcommand{\irr}{\operatorname{Irr}^*}
\newcommand{\per}{\operatorname{Per}_c}
\newcommand{\perb}{\overline{\operatorname{Per}}_c}
\renewcommand{\P}{\mathcal{P}}
\newcommand{\coh}{\operatorname{Coh}}
\renewcommand{\mod}{\operatorname{mod}}
\newcommand{\End}{\operatorname{End}}
\renewcommand{\L}{\mathcal{L}}
\newcommand{\M}{\mathcal{M}}
\renewcommand{\O}{\mathcal{O}}
\newcommand{\cc}{\mathbb{C}}
\newcommand{\zz}{\mathbb{Z}}
\newcommand{\rr}{\mathbb{R}}
\newcommand{\MM}{\mathfrak{M}}

\newcommand{\Ext}{\operatorname{Ext}}

\newcommand{\pic}{\operatorname{Pic}}
\newcommand{\Z}{\mathcal{Z}}
\newcommand{\F}{\mathcal{F}}
\newcommand{\I}{\mathcal{I}}
\newarrow{Biject} <--->

\begin{document}
\title{Counting invariants for the ADE McKay quivers}

\author{Amin Gholampour and Yunfeng Jiang}

\begin{abstract}
We consider the moduli space of the McKay quiver representations associated to the binary polyhedral groups $G\subset SU(2) \subset SU(3).$ The derived category of such representations is equivalent to the derived category of coherent sheaves on the corresponding ADE resolution $Y=G\hilb(\cc^3).$ Following the ideas of Nagao and Nakajima, by making particular choices of parameters in the space of stability conditions on the equivalent derived categories above, we recover Donaldson-Thomas (DT), Pand\-hari\-pande-Thomas (PT) and Szendroi (NCDT) moduli spaces. We also compute the Gromov-Witten (GW) partition function of $Y$ directly and express the result in terms of the root system of the associated ADE Dynkin diagram.
We then verify the conjectural GW/DT/NCDT-correspondence by assuming the DT/PT-correspondence. The Szendroi invariants are the same as the orbifold Donaldson-Thomas invariants for $\cc^3/G$ defined by Bryan. This will allow us to verify the Crepant Resolution Conjecture for the orbifold Donaldson-Thomas theory in this case.
\end{abstract}
%%% -----------------------------------------------------------------------
\maketitle
%%% ----------------------------------------------------------------------

\tableofcontents

\section{Introduction}
\subsection{Overview}
For a smooth Calabi-Yau threefold $Y$, there are three types of invariants that give rise to virtual count of curves in the given class $\beta\in H_2(Y,\zz)$:

\begin{itemize}
\item Gromov-Witten invariants $N_{g,\beta}$: obtained by virtual integration over $\M_g(Y,\beta)$, the moduli space of degree $\beta$ stable maps from curves of genus $g$ to $Y$, (see \cite{Beh-GWinvariants}),
\item Donaldson-Thomas invariants $I_{n,\beta}$: obtained by virtual integration over $I_n(Y,\beta)$, the moduli space of ideal sheaves of one dimensional subschemes of $\Z \subset Y$, of  holomorphic Euler characteristic $n$ and with $[\Z]=\beta$ (see \cite{MNOP1, Thomas-Casson}),
\item Pandharipande-Thomas invariants \footnote{Also called stable pair invariants.} $P_{n,\beta}$: obtained by virtual integration over $P_n(Y,\beta)$, the moduli space of stable pairs $(\F,s)$ consisting of a pure 1-dimensional sheaf $\F$ on $Y$ of holomorphic Euler characteristic $n$ and with $$[\text{support}(\F)]=\beta,$$ and a section $s:\O_Y\to \F$ with zero dimensional cokernel (see \cite{PT}).
\end{itemize}
We assemble these invariants into the following partition functions:
$$Z_{GW}^Y(\lambda;\mathbf{t})=\exp \left(\sum_{g=0}^\infty\sum_{\beta \neq 0} N_{g,\beta} \lambda^{2g-2} \mathbf{t}^\beta\right),$$
$$Z^Y_{DT}(q;\mathbf{t})=\sum_{n \in \zz}\sum_{\beta}I_{n,\beta}q^{n}\mathbf{t}^{\beta},$$
$$Z^Y_{PT}(q;\mathbf{t})=\sum_{n \in \zz}\sum_{\beta}P_{n,\beta}q^{n}\mathbf{t}^{\beta}.$$
Note that in the GW partition function the degree zero maps are not being considered. The exponential function is used to take the contribution of stable maps with disconnected domain curves into account.

These three types of invariants are conjecturally related to each other via simple but highly nontrivial formulas:

\begin{conjecture} \cite{MNOP1,MNOP2,PT} \label{con:DY/PT/GW}
$$Z_{GW}^Y(\lambda;\mathbf{t})=M(-q)^{-\chi(Y)}Z^Y_{DT}(q;\mathbf{t})=Z^Y_{PT}(q;\mathbf{t})$$
after the change of variable $q=-e^{i\lambda}$. Here $M(q)$ is the McMahon generating function for $3$-dimensional partitions, and $\chi(-)$ is the topological Euler characteristic.
\end{conjecture}
We refer to this conjecture as GW/DT/PT correspondence. The GW/DT correspondence has been proven when $Y$ is toric (see \cite{MOOP}). The DT/PT correspondence has recently been proven by Bridgeland and Toda
\cite{Bridgeland-hall,Toda-curve-counting} by using the works of Joyce \cite{joyce-motivic,joyceII} and Kontsevich and Soibelman \cite{kon-soib}, though some technical details yet to be checked. In this paper we assume the DT/PT correspondence.

We also consider another type of invariants called noncommutative Don\-ald\-son-Thomas invariants (NCDT). They were introduced by Szendroi via virtual integration over the moduli space of cyclic representations of a quiver $\mathfrak{G}$ with superpotential (see \cite{Szendroi,Joyce-Song}). We denote the NCDT invariant corresponding to the dimension vector $\mathbf{v}$ by $D_\mathbf{v}$. We assemble these invariants into the following partition function:
$$Z_{NCDT}^\mathfrak{G}(\mathbf{q})=\sum_\mathbf{v} D_\mathbf{v}\mathbf{q}^\mathbf{v}.$$

In this paper we follow closely the remarkable ideas used by Nagao and Nakajima in \cite{Nagao-toric,Nakajima-Nagao}. In these articles they associate a natural quiver to $Y$ in case $Y$ is a small toric Calabi-Yau threefold. By defining appropriate stability conditions on the moduli space of the quiver representations and proving wall-crossing formulas for the corresponding counting invariants, they are able to find a relation among  DT/PT/NCDT invariants. They use the equivalence of the abelian categories of perverse coherent sheaves and quiver representations in order to relate DT and PT invariants to the counting invariants of the quiver.

In the cases that we consider in this paper, $Y$ is not necessarily toric. However, we can define a natural $\cc^*$-action on $Y$ that allows us to express the DT/PT invariants in terms of the Euler characteristic of the moduli spaces. In contrast to the toric cases, an extra care needs to be taken as the $\cc^*$-fixed loci on $Y$ and on the moduli spaces involved are not necessarily isolated.

\subsection{The main results}

We consider the case that $Y=S\times \cc$, where $S$ is the ADE resolution of a germ of double point singularity $\cc^2/G$ for a finite subgroup $G\subset SU(2)$. We associate to $Y$ the McKay quiver arising naturally from the representation theory of $G$. We are able to relate DT/PT/NCDT invariants by means of similar wall-crossing formulas as in \cite{Nagao-toric,Nakajima-Nagao}. The first application of the wall crossing formula is a closed formula for the PT partition function of $Y=S\times \cc$. To express the result suppose that $R^+$ is the set of positive roots of the ADE root system attached to $S$. We use the identification of the root lattice of the root system with $H_2(Y,\zz)$. Let $N$ be the number of irreducible representation of $G$.

\begin{thm} \label{thm:PT}
$$Z_{PT}^{S\times \cc}(q;\mathbf{t})=\prod_{\beta \in R^{+}}\prod_{m=1}^{\infty}
(1-\mathbf{t}^{\beta}(-q)^{m})^{-m}.$$
\end{thm}
Assuming the DT/PT correspondence mentioned above, we get
\begin{corst} \label{cor:DT}
$$Z_{DT}^{S\times \cc}(q;\mathbf{t})=M(-q)^{N}\prod_{\beta \in R^{+}}\prod_{m=1}^{\infty}
(1-\mathbf{t}^{\beta}(-q)^{m})^{-m}.$$
\end{corst}

We calculate the GW partition function of $Y=S\times \cc$ by using localization and deformation techniques:
\begin{thm} \label{thm:GW}
$$Z_{GW}^{S\times \cc}(\lambda;\mathbf{t})=\prod_{\beta \in R^{+}}\prod_{m=1}^{\infty}
(1-\mathbf{t}^{\beta}(e^{i\lambda})^{m})^{-m}.$$
\end{thm}
Theorems \ref{thm:PT} and \ref{thm:GW} prove:
\begin{cor}
The GW/PT correspondence given in Conjecture \ref{con:DY/PT/GW} holds for the Calabi-Yau threefold $Y=S\times \cc$.
\end{cor}

The BPS invariants of the Calabi-Yau threefold $Y$ denoted by $n_{g,\beta}$ are conjecturally expressed in terms of GW invariants:

\[
\sum_{g=0}^\infty \sum_{\beta \neq 0} N_{g,\beta} q^\beta
\lambda^{2g-2}=\sum_{g=0}^\infty \sum_{\beta \neq 0} n_{g,\beta}
\sum_{d>0}
\frac{1}{d}\left(2\sin\left(\frac{d\lambda}{2}\right)\right)^{2g-2}q^{d\beta}.
\]
We have proven the following result
\begin{cor}
$$n_{g,\beta}=\begin{cases}-1 & \text{ if } g=0, \beta \in R^+ \\
0 & \text{ otherwise.}\end{cases}$$
\end{cor}

As discussed in \cite{Joyce-Song}, the NCDT invariants arising from the McKay quiver $\Q$ are the same as the orbifold DT invariants of $\cc^3/G$ defined by Bryan (see Appendix to \cite{Bryan-Young}) by means of the counting the configurations of colored boxes. Here $G$ can be either of the finite subgroups
\begin{enumerate} [i)]
\item $G \subset SU(2) \subset SU(3)$
\item $G \subset SO(3) \subset SU(3)$.
\end{enumerate}
The colors of boxes are in bijection with the set of irreducible representations of $G$, and hence the orbifold DT invariants depend on the variables $\mathbf{q}=(q_\rho)$ indexed by irreducible representation of $G$. Bryan has formulated a Crepant Resolution Conjecture relating the orbifold DT invariants of $\cc^3/G$ and the invariants of its crepant resolution given by $Y=G\hilb(\cc^3)$. This conjecture has been proven by Bryan and Young \cite{Bryan-Young} when $G$ is an abelian subgroup of $SO(3)$. We verify this conjecture for all $G$ in case i above by proving a closed formula for the partition function of the orbifold DT invariants (by assuming DT/PT correspondence):

\begin{thmst} \label{thm:DTorbifold}
\begin{align*}
Z_{DT}^{\cc^3/G}(\mathbf{q})&=Z_{NCDT}^\Q(\mathbf{q})=\\&M(-q)^{N}
\prod_{\beta\in R^{+}}\prod_{m=1}^{\infty}\left(1-\left(-\prod_{\rho\in \Irr(G)}q_\rho^{\dim \rho}\right)^{m}
 \prod_{\rho\in \irr(G)}q_\rho^{\beta_\rho}\right)^{-m} \\&\cdot
\prod_{\beta\in R^{+}}\prod_{m=1}^{\infty}\left(1-\left(-\prod_{\rho\in \Irr(G)}q_\rho^{\dim \rho}\right)^{m}
\prod_{\rho\in \irr(G)}q_\rho^{-\beta_\rho}\right)^{-m}.
\end{align*}
Here $\Irr(G)$ and $\irr(G)$ are respectively the sets of irreducible and nontrivial irreducible representations of $G$.
\end{thmst}

One of the crucial fact being used in this paper is that the the representation theory of an ADE Dynkin graph is of finite type (see Section \ref{sec:root systems}). The representation theory of a McKay quiver arising from a subgroup in case i above is related directly to the representation theory of the corresponding ADE Dynkin graph. For the case ii subgroups the situation is more involved as there is no such a direct relation to the finite type graphs. We will study the case ii above in a future work.

\subsection{Outline}
In Section \ref{sec:ade-quiver}, we fix the notation and setup for the rest of the paper. We review the construction of the ADE McKay quivers and their superpotentials, and discuss the equivalence of the category of the quiver representations and a  category of perverse coherent sheaves on $Y$.

In Section \ref{sec:stability}, we define the space of stability conditions on the categories of framed representation of the ADE McKay quivers, define the associated moduli space and the counting invariants, and finally determine the chambers corresponding to the DT and PT invariants.

In Section \ref{sec:wall crossing}, we relate the walls corresponding to non-generic stability parameters to the ADE root systems attached to the McKay quivers. We show that the non generic walls are in correspondence with the set of positive roots of the affine root system. In this section we prove the wall crossing formula, Theorem \ref{thm:DTorbifold}, and the DT Crepant Resolution Conjecture.

In Section \ref{sec:GW}, we prove Theorem \ref{thm:GW}, GW/PT correspondence for the case at hand, and Theorem \ref{thm:PT}.

In Section \ref{sec:example}, as an illustration, we consider the concrete example of $G=\mathbb{D}_{12}$, the binary dihedral group in 12 elements.

\section{McKay Quivers associated to an affine ADE Dynkin diagram}\label{sec:ade-quiver}
%%%--------------------------------------------------------------------

\subsection{McKay quivers} \label{sec:mckay quiver}
Let $G$ be a finite subgroup of $SU(2)$. It is well-known that $G$ admits an ADE classification, and it falls into one of three classes of cyclic groups, binary dihedral, and the binary versions of groups of symmetries of the platonic solids.

In this paper we consider the action of $G$ on $\cc^3$ via the natural embedding $SU(2) \subset SU(3)$.
Let $$X=\cc^3/G$$ and $Y=G\hilb(\cc^3)$. Then the Hilbert-Chow morphism $$\pi:Y \to X$$ is the Calabi-Yau resolution of $X$. Note that the resolution $\pi$ differs from the classical minimal resolution $S \to \cc^2/G$ by only a trivial factor of $\cc$. Hence by the McKay correspondence (see \cite{McKay}), the fiber of $\pi$ over the origin gives rise to an ADE Dynkin diagram:

\[
\begin{diagram}
\left\{\begin{array}{l}
\text{Non-trivial}\\
\text{irreducible}\\
\text{$G$-representations}
\end{array} \right\}
&\rBiject&
\left\{\begin{array}{l}
\text{irreducible}\\
\text{components of} \\
\pi^{-1}(0)
\end{array} \right\}
&\rBiject &
\left\{\begin{array}{l}
\text{Simple roots}\\ \text{of an ADE}\\
\text{root system}
\end{array} \right\}.
\end{diagram}
\]
Denote the irreducible component of $\pi^{-1}(0)$ corresponding to $\rho\in \irr(G)$ by $C_\rho$. Then $\{C_\rho\}$ represents a basis for $H_2(Y,\zz)$.

To any $G$ as above one associates a natural quiver with superpotential, called the McKay quiver.
Let $\Irr(G)$ be the set of irreducible representations of $G$ and define $$\irr(G)=\Irr(G)-\{\rho_0\}$$ where $\rho_0$ is the trivial irreducible representation. We define $N=|\Irr(G)|$.
The McKay quiver $\Q$ has a vertex for each irreducible representation $\rho \in \Irr{G}$. We label the vertex corresponding to $\rho$ by the same letter. Two vertices $\rho$ and $\rho'$ are joined by a directed edge $\rho \to {\rho'} $ if $\rho'$ appears in the decomposition of the representation $\rho \otimes \mathcal{V}$ into irreducibles. Here $\mathcal{V}$ is the natural 3-dimensional representation of $$G \subset SU(2) \subset SU(3).$$ See Figure \ref{fig:ADE} for the list of possible McKay quivers we consider in this paper. Let $\D$ be a Dynkin quiver (an extended ADE diagram with an arbitrary orientation). Then $\Q$ is obtained by adding a loop to each vertex of $\overline{\D}$ (the double of $\D$).
\begin{figure} [htp]
\label{fig:ADE}
\begin{center}
\includegraphics{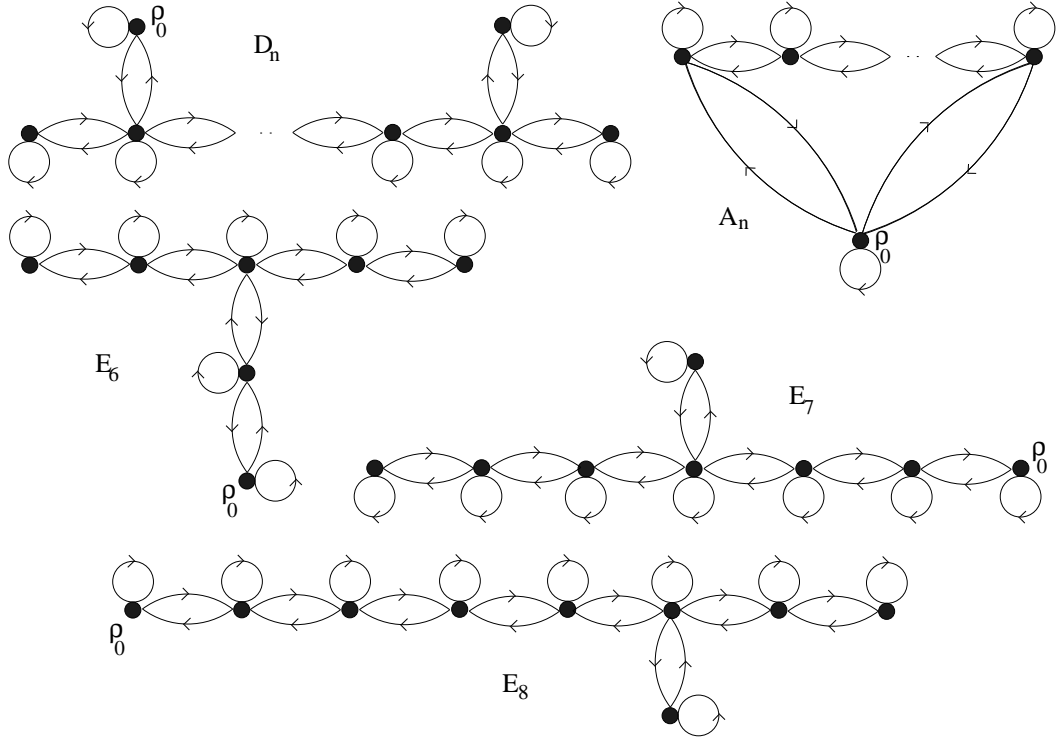}
\caption{ADE McKay quivers}
\end{center}
\end{figure}

We label the loop $\rho\to \rho$ by $l_\rho$ and the edge $\rho \to \rho'$ by $r_{\rho \rho'}$. Let $\cc \Q$ be the associated path algebra of $\Q$. The superpotential $$W \in \cc \Q/[\cc \Q,\cc \Q] $$ of the McKay quiver is given by the sum of cubic monomials one for each directed cycle $$\rho \rightarrow \rho'\rightarrow \rho'' \rightarrow \rho$$ of $\Q$.
One can see that (see \cite{BSW-superpotentials}).
\begin{equation}\label{equ:superpotential}
W=\sum_{\tiny{\begin{array}{c}\rho,\,\rho' \in \Irr(G),\\ \text{$r_{\rho\rho'}$ is an edge of $\D$} \end{array}}}(r_{\rho'\rho}r_{\rho \rho'}l_{\rho}-l_{\rho'}r_{\rho \rho'}r_{\rho' \rho}).
\end{equation}

Let $I< \cc \Q$ be the two-sided ideal generated by the partial derivatives of $W$ with respect to $l_\rho$'s and $r_{\rho\rho'}$'s, and define $$A=\cc \Q/I.$$ Denote by $\Pi$ the preprojective algebra $$\Pi=\frac{\cc \overline{\D}}{\left(\displaystyle \sum_{\text{$r_{\rho\rho'}$ is an edge of $\D$}}r_{\rho\rho'}r_{\rho'\rho}-r_{\rho'\rho}r_{\rho\rho'}\right)}$$ associated to $\D$. There is a natural surjection $A\to \Pi$, obtained by setting $l_\rho=0$ for any $\rho \in \Irr(G)$.

We denote by $\mod(A)$ the abelian category of finitely generated (right) $A$-modules. $\mod(A)$ is equivalent to the
abelian category of finite dimensional $(\Q,I)$-representations. Let $D^{b}(\mod(A))$ be the associated bounded derived category.

%%%----------------------------------------
\subsection{Perverse coherent system} \label{sec:perverse system}
Let $D^{b}_c(\coh(Y))$ be the bounded derived category of coherent sheaves on
$Y$, whose objects have compactly supported cohomologies, and let$$\per(Y/X)\subset D^b_c(\coh(Y))$$ be the full subcategory
of perverse coherent sheaves of perversity -1, in the sense of Bridgeland \cite{Bridgeland-flop}.

Following Van Den Bergh \cite{Bergh-flop}, for any $\rho \in \irr(G)$ let $\L_\rho \in \pic(Y)$ be such that $$\L_\rho|_{C_{\rho'}}=\delta_{\rho \rho'} \text{ for any } \rho' \in \irr(G).$$ Define $\M_\rho$ to be the middle term of the short exact sequence $$0\longrightarrow \O_{Y}^{\dim \rho-1}\longrightarrow \M_\rho \longrightarrow \L_\rho \longrightarrow 0.$$ Let $\M_{\rho_0}=\O_Y$ for the trivial irreducible representation $\rho_0$ and define the projective generator $$\P=\bigoplus_{\rho \in \Irr(G)} \M_\rho.$$

We also consider the abelian category of perverse coherent systems, denoted by $\perb(Y/X)$. It has as objects the triples $(\F,W,s)$ of an $\F\in \per(Y/X)$, a vector space $W$, and a morphism $$s: W\otimes_{\cc}\O_{Y}\rightarrow \F.$$ Morphisms in $\perb(Y/X)$ are defined in an obvious way (see \cite{Nakajima-Nagao}).

\begin{defn} We define a new quiver $\Qt$ by adding a vertex $\infty$ and an edge $$\infty \to \rho_0$$ to the McKay quiver $\Q$ labeled by $r_\infty$. We denote by $\A=\cc \Qt/I$ the path algebra of $\Qt$ with the same relations as for $A$. We denote by $\mod(\A)$ the abelian category of finite (right) $\A$-modules.
\end{defn}

The following theorem links the theory of quiver representations to the theory of the perverse sheaves. The details of the proof can be found in \cite{Ginzburg-calabi, Bergh-flop, Nagao-toric, Nakajima-Nagao}.

\begin{thm} \label{thm:equivalence of categories}
We have an isomorphism of algebras $$A \cong \End_Y(\P),$$ and equivalence of categories
$$\mod(A) \cong \per(Y/X)\; \text{ and }\; \mod(\A) \cong \perb(Y/X).$$
For any $\F \in \per(Y/X)$, let $V=(V_\rho)_{\rho \in \Irr(G)}$ be the corresponding $A$-module then
$$V_\rho \cong H^{0}(\F\otimes \M^{\vee}_\rho).$$
\end{thm}

%%%---------------------------------------------------------------------------------------
\section{Stability conditions} \label{sec:stability}

%%%---------------------------------------------------------------------------------------
\subsection{$\theta$-stability conditions.}
In this section we define a stability condition on $\mod(\A)$. Let $$\zeta=(\zeta_\rho)_{\rho \in \Irr(G)} \in \rr^N$$ be a sequence of real numbers, and take $\zeta_{\infty} \in \rr$.
For a finite dimensional $A$-module $$V=(V_\rho)_{\rho \in \Irr(G)}$$ and a finite vector space $V_{\infty}$
define
$$\theta_{\zetat}(V,V_{\infty})=\frac{\zeta \cdot \dim V+\zeta_{\infty}\dim V_{\infty}}
{\sum_\rho \dim V_\rho+\dim V_{\infty}} $$ where $\zetat=(\zeta,\zeta_{\infty})$.

\begin{defn}\label{defn:theta-stability}
\begin{enumerate} [i)]
\item An $\A$-module $\V=(V,V_{\infty})$ is called $\theta_{\zetat}$-(semi)stable if for any nonzero $\A$-submodule
$\V' \subset \V$,
$$\theta_{\zetat}(\V')(\leq)< \theta_{\zetat}(\V).$$

\item For a given $\zeta \in \rr^N$, An $\A$-module $\V=(V,V_{\infty})$ with $V_{\infty}\neq 0$ is called $\zeta$-(semi)stable if it is $\theta_{\zetat}$-(semi)stable for $$\zetat=(\zeta,\zeta_{\infty})$$ where $\zeta_{\infty}$ is chosen so that $\theta_{\zetat}(\V)=0$.

\item For a given $\zeta \in \rr^N$, an $A$-module $V$ is called $\zeta$-(semi)stable if the $\A$-module $(V,0)$ is $\zeta$-(semi)stable.
\end{enumerate}
\end{defn}

\subsection{The moduli space and the counting invariants}
Let $\zeta \in \rr^{N}$ be a stability parameter and let $$\mathbf{v} \in (\zz_{\ge 0})^{N}.$$
We denote by $\ms$ the moduli space of $\zeta$-stable $\A$-modules $\V=(V,\cc)$ with $\dim V=\mathbf{v}$.

$\ms$ admits a symmetric perfect obstruction theory in the sense of \cite{Behrend-microlocal} as it is the critical locus of the trace of the superpotential function, $w=\operatorname{tr}(W)$, defined on a smooth variety $$\mathcal{S}=\left(\prod_{\rho\to \rho' \text{ is an edge in $\Q$}}\operatorname{Hom}(V_\rho,V_{\rho'})\right)\times V_{\rho_0}.$$ (see \cite{Szendroi,Joyce-Song}). Hence, there is a virtual fundamental class $[\ms]^{vir}$ of virtual dimension zero. The counting invariants are defined by
$$\#^{vir}\ms=\mbox{deg}([\ms]^{vir}).$$

From \cite{Behrend-microlocal}, there exists an integer value constructible function
$$\nu_{\MM_{\zeta}}: \ms \longrightarrow \zz$$
such that the counting invariants are given by the weighted Euler characteristic:
$$\#^{vir}\ms=\sum_{n\in\zz}n\cdot\chi(\nu^{-1}(n))
=\chi(\ms, \nu_{\MM_{\zeta}}).$$

Let $\mathbf{q}=(q_\rho)_{\rho \in \Irr(G)}$ and $$\mathbf{q}^{\mathbf{v}}=\prod_{\rho}q_\rho^{v_\rho}.$$
Define the generating function for the counting invariants as:
\begin{equation}\label{equ:gen-1}
Z_{\zeta}(\mathbf{q})=\sum_{\mathbf{v}}\#^{vir}\ms\mathbf{q}^{\mathbf{v}}
=\sum_{\mathbf{v}}\chi(\ms, \nu_{\MM_{\zeta}})\mathbf{q}^{\mathbf{v}}.
\end{equation}

For each point $P\in\ms$, the parity of the Zariski tangent space at $P$ is the same as the parity of $v_{\rho_0}$.
This can be seen as follows: The point $P\in\ms$
corresponds to an $\A$-module $\tilde{V}$, by \cite[Corollary 2.5.3]{Szendroi}, we have
$$\dim T_{P}\ms =\sum_{\rho}v_\rho^2+\sum_{[v_\rho\rightarrow v_\rho']}2v_\rho v_\rho'+v_{\rho_0}-\sum_{\rho}v_\rho^2\equiv v_{\rho_0} (\mbox{mod}~  2).$$ The first three terms add up to the dimension of $\mathcal{S}$, and the last term is the dimension of the linear algebraic group acting on $\mathcal{S}$.
We can then express the Behrend function as follows (see \cite{Behrend-microlocal}):
\begin{equation} \label{equ:milnor fiber}
\nu_{\MM_{\zeta}}(P)=(-1)^{v_{\rho_0}}(1-\chi(F_{P})),
\end{equation}
where $F_P$ is the Milnor fiber of $w$ at $P$.

We consider the following $\cc^*$-action on $\A$. We let the weight of the action on $r_{\rho \rho'}$, $\l_\rho$, and $r_\infty$ be 1, -2, and 0, respectively. Note that the $\cc^*$-action preserves the superpotential $W$ given in (\ref{equ:superpotential}) and hence the induced action on $\mathcal{S}$ and $\ms$ has this property that the symmetric obstruction theory on $\ms$ is $\cc^*$-equivariant \cite{Szendroi, Behrend-microlocal}. Denote by $\msc$ the $\cc^*$-fixed locus of $\ms$. It is not hard to see that if an $\A$-module $\V$ is $\cc^*$-invariant then $l_\rho \V=0$ for any $\rho \in \Irr(G)$. We can furthermore prove:

\begin{lem}  \label{lem:smooth fixed locus}
The $\cc^*$-fixed locus, $\msc$, is smooth.
\end{lem}
\begin{proof}
Suppose $\V \in \msc$. Denote by $L_\rho$ and $R_{\rho\rho'}$ the linear homomorphisms of $\V$ corresponding to $\l_\rho$ and $r_{\rho\rho'}$. As we mentioned above $L_\rho=0$ for any $\rho \in \Irr(G)$. Moreover, since the weight of the action on all $R_{\rho\rho'}$ is 1, given any cycle $$r_{\rho_1\rho_2}r_{\rho_2\rho_3}\dots r_{\rho_n\rho_{n+1}}$$ in $\Q$ (with $ \rho_{n+1}=\rho_1$) the composition $$R_{\rho_1\rho_2}\circ R_{\rho_2\rho_3}\circ \cdots \circ R_{\rho_n\rho_{n+1}}$$ is not $\cc^*$-invariant unless $R_{\rho_i\rho_{i+1}}=0$ for some $i$. In particular, at least one of $R_{\rho\rho'}$ or $R_{\rho'\rho}$ is zero. This shows that $\msc$ is possibly a disjoint union of the moduli spaces of stable representations of hereditary algebras, and hence is smooth (see \cite{king}).
\end{proof} Using this lemma, the restriction of the Behrend function to the fixed locus is easily evaluated:
\begin{prop} \label{prop:behrend function on fixed locus}
For any $P \in \msc$, $$\nu_{\MM_{\zeta}}(P)=(-1)^{v_{\rho_0}}.$$
\end{prop}
\begin{proof}
We know that the fixed locus of the Milnor fiber, $F_P^{\cc^*}$, is the Milnor fiber of $w|_{\mathcal{S}^{\cc^*}}$. Moreover the fixed locus of the critical locus of $w$ is the critical locus of $w|_{\mathcal{S}^{\cc^*}}$, which is $\msc$.
By the smoothness of $\msc$ (Lemma \ref{lem:smooth fixed locus}), $\chi(F_P^{\cc^*})=0$, and hence the claim follows from (\ref{equ:milnor fiber}).
\end{proof} We will use the following corollary in Section \ref{sec:wall crossing}:

\begin{cor} \label{cor:euler char of moduli space}
$$Z_{\zeta}(\mathbf{q})=\sum_{\mathbf{v}} (-1)^{v_{\rho_0}}\chi(\ms)\mathbf{q}^{\mathbf{v}}.$$
\end{cor}
\begin{proof}
The claim follows immediately from Proposition \ref{prop:behrend function on fixed locus} and (\ref{equ:gen-1}).
\end{proof}

%%%--------------------------
\subsection{Chambers corresponding to DT/PT theories} \label{sec:chambers DT-PT}

Let $$r=\sum_{\rho \in \Irr^*(G)} \dim \rho,$$ and $\epsilon >0$ be a sufficiently small real number.

Define the stability parameters $\zeta^{im}$ and $\zeta^{im,\pm}$ by
$$\zeta^{im}_\rho= \begin{cases} -r & \rho=\rho_0,\\
 1 & \text{otherwise,}\end{cases} \text{\hspace{.5cm} and \hspace{.5cm}} \zeta^{im,\pm}_\rho= \begin{cases} -r\pm \epsilon & \rho=\rho_0,\\
 1 & \text{otherwise.}\end{cases}$$
$\zeta^{im}$ lies on the hyperplane in $\rr^N$ identified by the vector $$(-r,\underbrace{1,1,\dots,1}_{N-1}).$$ $\zeta^{im,\pm}$ lie on the two sides of this hyperplane.

\begin{prop}\label{prop:rF}
For any $\F\in \per(Y/X)$, $$\chi(\F)=\dim H^0(\F),$$ and moreover, the coefficient of the degree one term of the Hilbert polynomial of $\F$ with respect to the polarization $$\L=\bigotimes_{\rho \in \irr(G)}\L_\rho$$ ($\L_\rho$'s were defined in Section \ref{sec:perverse system}) is given by
$$ r \dim H^{0}(\F)- \sum_{\rho \in \irr(G)}\dim H^{0}(\F\otimes \M^{\vee}_\rho).$$
\end{prop}
\begin{proof}
The first equality follows from the definition of $\per(Y/X)$. To prove the second equality, note that the Hilbert polynomial of any $\F \in \per(Y/X)$ with respect to any polarization is linear, and hence the coefficient of the degree one term with respect to $\L_\rho$ is given by \begin{align*}\chi(\F)-\chi(\F\otimes \L^{-1}_\rho)&=(\dim \rho)\chi(\F)-\chi(\F \otimes \M^{\vee}_\rho)\\&=(\dim \rho)\dim H^0(\F)-\dim H^0(\F \otimes \M^{\vee}_\rho).\end{align*}
Since the leading term of the Hilbert polynomial is linear with respect to the tensor products of polarizations, this formula concludes the proposition.
\end{proof}

The following proposition together with Theorem \ref{thm:equivalence of categories} and Proposition \ref{prop:rF} are used to identify the moduli spaces of $\zeta^{im,\pm}$-stable $\A$-modules with the certain moduli spaces of stable pairs or ideal sheaves. The proof is given in \cite[Section 2]{Nakajima-Nagao}, and is not repeated here.

\begin{prop}\label{prop:DT-PT chambers} Suppose that $\V=(V,\cc)$ is the $\A$-module corresponding to the perverse coherent system $$(\F,s,\cc) \in \perb(Y/X),$$ where $V$ is as given in Theorem \ref{thm:equivalence of categories}. If $\epsilon$ in the definition of $\zeta^{im,\pm}$ is chosen to be less than $1/\dim H^0(F)$ then
\begin{enumerate}[i)]
\item If $\V$ is $\zeta^{im,-}$-stable, then $\F$ is a coherent sheaf and $s$ is surjective. In other words $(\F,s,\cc)$ is equivalent to an ideal sheaf
$\I \in \coh(Y)$ fitting into the following exact sequence
\begin{diagram} &0& \rTo &\I& \rTo &\O_Y& \rTo^s &\F&\rTo &0.&\end{diagram}
On the other hand, if for $(\F,s)$, $s$ is surjective, then $(\F,s)$ is $\zeta^{im,-}$-stable.
\item If $\V$ is $\zeta^{im,+}$-stable, then $\F$ is a pure sheaf of dimension one and the cokernel $\mbox{coker}(s)$  is $0$-dimensional. In other words $(\F,s)$ is a stable pair defined in \cite{PT}.
On the other hand, if $(\F,s)$ is a stable pair, then $\V$ is $\zeta^{im,+}$-stable.
\end{enumerate}
\end{prop}

\begin{cor} \label{cor:PT/DT chambers}
We have the following isomorphisms of the moduli spaces:
$$P_n(Y,\beta)\cong \MM_{\zeta^{im,+}}(\mathbf{v}),$$
$$I_n(Y,\beta) \cong\MM_{\zeta^{im,-}}(\mathbf{v}),$$
where $n=v_{\rho_0}$, and $$\beta=\sum_{\rho\in \irr(G)} (v_{\rho_0}\dim \rho-v_\rho)[C_\rho].$$
($C_\rho$'s were defined in Section \ref{sec:mckay quiver}.)
\end{cor}

%
%%%------------------------------------------------------------------------
%%------------------------------------------------------------------------
\section{Root systems, stability chambers, and walls} \label{sec:wall crossing}

%%%------------------------------------------------
\subsection{Root systems} \label{sec:root systems}
Following \cite{Kac-root}, in this section, we review some of the terminologies in the theory of root systems that we will use in the next sections. Let $\Gamma$ be a connected graph with $K+1$ vertices $\{0,\cdots,K\}$. Denote by $b_{ij}$ the number of edges connecting the vertices $i$ and $j$, if $i\neq j$, and
twice the number of loops at the vertex $i$ if $i=j$. Let $\{\alpha_i\}_{i=0,1,\dots,K}$ be the standard basis for
$\zz^{K+1}$.  Define a bilinear form on $\zz^{K+1}$ by
$$(\alpha_i, \alpha_i)=\delta_{ij}-\frac{1}{2}b_{ij}.$$
Let $Q(\alpha)$ be the associated quadratic form.
The element $\alpha_i$ is called a \emph{fundamental root} if there is no loops at the vertex $i$. Let $\mathfrak{P}$ be the set of fundamental roots. For a fundamental root $\alpha$, the \emph{fundamental reflection}
$\varsigma_{\alpha}\in Aut(\zz^{K+1})$ is defined by
$$\varsigma_{\alpha}(\lambda)=\lambda-2(\lambda,\alpha)\alpha$$
for $\lambda\in \zz^{K+1}$.
The reflection $\varsigma_{\alpha}$ satisfies the relations
$$\varsigma_{\alpha}(\alpha)=-\alpha$$ and $$\varsigma_{\alpha}(\lambda)=\lambda$$ if $(\lambda,\alpha)=0$.
The Weyl group, $\mathfrak{W}(\Gamma)$, associated to $\Gamma$ is generated by all the fundamental reflections.
\begin{defn}
The set of real roots is defined by
$$R^{re}(\Gamma)=\bigcup_{\omega \in \mathfrak{W}(\Gamma)}\omega(\mathfrak{P}).$$
\end{defn}

The subset of $\zz^{K+1}$ defined as $$\mathfrak{S}=\{\alpha\in\zz_{\ge 0}^{K+1}\setminus\{0\}| (\alpha,\alpha_i)\leq 0 \text
{ for all }\alpha_i\in \mathfrak{P}\}$$ is call the \emph{fundamental set}.
\begin{defn}
The set of imaginary roots is defined by
$$R^{im}(\Gamma)=\bigcup_{\omega\in \mathfrak{W}(\Gamma)}\omega(\mathfrak{S}\cup(-\mathfrak{S})).$$
\end{defn}
\begin{rmk}
One can verify easily that $(\alpha,\alpha)=1$ if
$\alpha\in R^{re}(\Gamma)$ and $(\alpha,\alpha)\leq 0$ if
$\alpha\in R^{im}(\Gamma)$.
\end{rmk}
\begin{defn}
The root system associated to $\Gamma$ is defined as $$R(\Gamma)=R^{re}(\Gamma)\cup R^{im}(\Gamma).$$ The set of \emph{positive roots} is $$R^{+}(\Gamma)=R(\Gamma)\cap \zz_{\ge 0}^{K+1},$$ and $R^{+,re}(\Gamma)$ and $R^{+,im}(\Gamma)$ are the corresponding positive real and imaginary roots, respectively.
\end{defn}

It is well known that an ADE Dynkin diagram has finitely many positive roots and has no imaginary roots. The extended ADE Dynkin diagrams (or affine ADE diagrams) are \emph{tame} type graphs according to \cite{Kac-root}. Any tame type graph has infinitely many positive roots. In the case $\Gamma$ is an ADE (respectively extended ADE) Dynkin diagram we denote the corresponding root system by $R$ (respectively $\R$). Dropping $\Gamma$ from the notation will cause no confusion. The following theorem is well known to the experts. The proofs of the statements can be found in \cite{Crawley-kleinian, FMV-affine, Kac-root, king}:

\begin{thm}\label{thm:root-representation}
Let $\D$ be a Dynkin quiver (see Section \ref{sec:mckay quiver}). There exists an indecomposable $\D$-representation of dimension vector $\ah$ if and only if $\ah \in \R^+$. For each positive real root $\ah \in \R^{+, re}$ and a generic $\zeta \in \rr^N$ with $\zeta \cdot \ah=0$, there is a unique, up to isomorphism, $\zeta$-stable $\D$-representation $U$ with dimension vector $\ah$. Moreover, $U$ gives rise to the unique $\zeta$-stable $\Pi$-module of dimension vector $\ah$.
\end{thm}

For a given ADE Dynkin diagram let $\beta \in R$, and suppose that $\ah\in \R$ be a root in the extended version of the same diagram. Define $\ah'=\ah+\beta$ to be given by $$\ah'_\rho=\begin{cases}\ah_\rho+\beta_\rho & \text{ if } \rho\in \irr(G),\\\ah_{\rho_0} & \text{ if } \rho=\rho_0.\end{cases}.$$

The following proposition is proven in \cite{Kac-root}:

\begin{prop}\label{relation-roots}
Let $R$ and $\R$ be the root systems of the ADE and extended ADE Dynkin diagrams associated to $G$. Then
$$\R^{im}=\{m\ah^{im}| m\in \zz \},$$ where $$\ah^{im}=(\dim \rho)_{\rho \in \Irr(G)}.$$ Moreover,
 $$\R^{+,re}=\{\ah+\beta|\ah\in R^{im},\; \beta\in R^+\}.$$
\end{prop}

\begin{example}\label{E7}
In the $\widehat{E}_{7}$ case, the imaginary root $\ah^{im}$ is
$$
\begin{array}{ccccccc}
&&&2&&& \\
1&2&3&4&3&2&1
\end{array}
$$
$$\beta=\begin{array}{ccccccc}
&&2&&& \\
2&3&4&3&2&1
\end{array} $$ is a positive root for $E_7$. Then
$$m\ah^{im}+\beta=
\begin{array}{ccccccc}
&&&2m+2&&& \\
m&2m+2&3m+3&4m+4&3m+3&2m+2&m+1
\end{array}
$$
is a positive real root of $\widehat{E}_7$.
\end{example}

\subsection{Cutting lemma and non-generic walls}
The following lemma is an analog of \cite[Lemma 3.4]{Nakajima-Nagao} and \cite[Proposition 2.9]{Nagao-toric}:

\begin{lem}\label{lem:cut-lemma}
Let $U$ be a nonzero $\zeta$-stable $A$-module with $l_\rho U=0$ for any $\rho \in \Irr(G)$, then $\dim U \in \R^+$. If $\ah \in \R^{+,re}$, and $\zeta$ is a generic point in $\rr^N$ with $\zeta \cdot \ah=0$ then there exists a unique, up to isomorphism,  $\zeta$-stable $A$-module $U$ with $\dim U= \ah$ and $l_\rho U=0$ for any $\rho \in \Irr(G)$.
\end{lem}
\begin{proof}
This is a direct corollary of Theorem \ref{thm:root-representation} by noting that an $A$-module with $l_\rho U=0$ for any $\rho \in \Irr(G)$ can be regarded as a $\Pi$-module.
\end{proof}

A stability parameter $$\zeta=(\zeta_\rho)_{\rho \in \Irr(G)}\in \rr^N$$ is said to be generic if
$\zeta$-semistability is equivalent to $\zeta$-stability. We are now ready to classify the non-generic stability parameters for the $\cc^*$-invariant $\A$-modules.
\begin{defn}
A hyperplane in $\rr^N$ corresponding to a non-generic stability parameter is called a wall.
\end{defn}

\begin{prop}\label{prop:walls}
The set of non-generic stability parameters for the $\cc^*$-invariant $\A$-modules is given by the union of the following walls:
$$W_{\ah}=\{\zeta\in \rr^{N}| \zeta\cdot \ah=0, \ah\in \R^{+}\}.$$
\end{prop}
\begin{proof}
For a $\cc^*$-invariant strictly $\zeta$-semistable $\A$-module $\V=(V,\cc)$ one has the Jordan-H{\"o}lder filtration
$$0=\V^{n+1}\subset \cdots \subset \V^{1}\subset \V^{0}=\V $$ where $\V^i/\V^{i+1}$ is stable and $$\theta_{\zetat}(\V^i/\V^{i+1})=0$$ for all $0<i<n+1$. $\zetat$ is as defined in Definition \ref{defn:theta-stability} ii). By the $\cc^*$-invariance of $\V$ we have $l_\rho \V^i=0$ for any $0\le i \le n+1$ and $\rho \in \Irr(G)$.
Note that  $n\geq 1$ since $\tilde{V}$ is not $\zeta$-stable.  So at least one of
$(\V^{l}/\V^{l+1})_{\infty}$ is zero. Hence there exist a nonzero
$\theta_{\tilde{\zeta}}$-stable $\A$-module $\tilde{V}'=(V',V_{\infty}')$ such that
$V_\infty'=0$, $\dim V' \cdot \zeta=0$, and $\l_\rho V'=0$. By Lemma \ref{lem:cut-lemma}, $\dim V' \in \R^+$ and the proposition follows.
\end{proof}

\begin{defn}
The wall $W_{\ah^{im}}$ is called the imaginary wall in $\rr^{N}$.
\end{defn}

\begin{rmk}
Let $\zeta^{im}$ be as in Section \ref{sec:chambers DT-PT}, then $\zeta^{im}\cdot \ah^{im}=0$. From Proposition \ref{prop:DT-PT chambers} one can see that the wall $W_{\ah^{im}}$ is the wall separating the Donaldson-Thomas and Pandharipande-Thomas theories. Our wall crossing formula in this section does not include crossing the imaginary wall.
\end{rmk}

%%%----------------
\subsection{Wall-crossing formula} \label{sec:wall-formula}
The non-generic stability parameters (for the $\cc^*$ invariant $\A$-modules) lie on a finite number of walls as in Proposition \ref{prop:walls}.
Each component of the set of generic stability parameters is called a chamber. Inside a chamber the moduli spaces $\msc$ are isomorphic. In this section we relate the corresponding invariants of two chambers separated by a non imaginary wall.

Fix an $\alpha\in \R^{+,re}$ and let $\zeta_\alpha \in W_\alpha$
be a generic point. For a $0<\epsilon \ll 1$, let $\zeta_\alpha^{\pm}=(\zeta_{\alpha,\rho} \pm \epsilon)$. Then
$\zeta_\alpha^{\pm}$ lie in two chambers on two sides of the wall $W_\alpha$.

Let $U$ be the unique $\zeta_\alpha$-stable $A$-module of dimension vector $\alpha$ given by Lemma \ref{lem:cut-lemma}.  Note that by definition $\zeta_\alpha \cdot \dim(U)=0$.

We have the following identities:

\begin{prop}\label{prop:ext}
\begin{enumerate} [i)]
\item $\Ext^{1}_{A}(U,U)\cong\cc.$
\item If $\tilde{V}=(V,\cc)$ is a $\zeta_{\alpha}$-stable $\tilde{A}$-module, then
$$\dim \Ext^{1}_{\tilde{A}}(U,\tilde{V})-\dim \Ext^{1}_{\tilde{A}}(\tilde{V},U)=\dim(U_{\rho_0}).$$
\end{enumerate}
\end{prop}
\begin{proof}
\begin{enumerate} [i)]
\item As mentioned earlier, $U$ can be regarded as a $\Pi$-module. By \cite{Crawley-fibers}, and using the fact that $\dim U \in \R^{+,re}$, we get that $\Ext^1_{\Pi}(U,U)=0$. By using Serre duality and the stability of $U$ we have
    $$\Ext^0_\Pi(U,U)\cong\Ext^2_\Pi(U,U)\cong\cc.$$ Moreover, $\Ext^q_\Pi(U,U)=0$ for $q>2$.
    The following identities can be easily verified:
    $$\Ext^q_A(\Pi,U)=\begin{cases}U &\text{ if } q=0,1\\ 0 & \text{ otherwise.}\end{cases}$$
    Now the claim follows the spectral sequence $$\Ext^p_{\Pi}(U,\Ext^q_A(\Pi,U)) \Rightarrow \Ext^{p+q}_A(U,U)$$  corresponding to the natural change of rings from $A$ to $\Pi$ obtained by setting all $l_\rho$'s equal to zero.
\item This part follows at once from \cite[Theorem 7.5]{Joyce-Song} by noting that the $\zeta_\alpha$-stability of $U$ and $\V$ implies that $$\operatorname{Hom}_{\tilde{A}}(U,\tilde{V})=\operatorname{Hom}_{\tilde{A}}(\tilde{V},U)=0.$$
\end{enumerate}
\end{proof}

For a positive integer $m$, let $U_m$ be the unique indecomposable $A$-module
which is obtained by $m-1$ times successive extensions of $U$, i.e.
there is an exact sequence
$$0\longrightarrow U_{m-1}\longrightarrow U_m\longrightarrow U\longrightarrow 0$$
for each $m$. This is possible by part i) in Proposition \ref{prop:ext}.

As in \cite{Nakajima-Nagao, Nagao-toric}, we consider the following stratifications for the $\cc^*$-fixed loci of $\msa$ and $\mspn$.
\begin{itemize}
\item For $L\in \zz_{\geq 0}$, let $\msca_L^+$ be
the subscheme of $\msca$ consisting of the points $\V$ such that $\dim \Ext^{1}(U,\V)=L$.
Similarly, let $\msca_{L}^-$ be
the subscheme of $\msca$ consisting of the points $\V$ such that $\dim \Ext^{1}(\V,U)=L$.
\item For the sequence $(d)=d_1,d_2,\dots$ of nonnegative integers with $d_i \neq 0$ for only finitely many $i$'s, let
$\msp_{(d)}$ (respectively $\msn_{(d)}$) be the subscheme of $\msp$ (respectively $\msn$) containing the points $\V'$ for which there exists a unique $\cc^*$-invariant stable $\A$-module $\V$ fitting into the short exact sequence
$$0\longrightarrow \tilde{V}\longrightarrow \tilde{V}'\longrightarrow \oplus_{m'\geq 1}(U_m')^{\oplus d_{m'}}\longrightarrow 0, $$
(respectively $$0\longrightarrow\oplus_{m'\geq 1}(U_m')^{\oplus d_{m'}}\longrightarrow \tilde{V}'\longrightarrow \tilde{V}\longrightarrow  0 \text{).} $$
\end{itemize}
Then one can define a canonical morphism
$$\phi_{\pm}: \mscpn_{(d)} \longrightarrow \msca,$$ with  $\phi_{\pm}(\V')=\V$. Note that
$\mathbf{v}'=\mathbf{v}+\sum_{m}md_m\cdot \dim(U)$.
Define $$\mscpn_{(d),L}=\phi_{\pm}^{-1}\left(\msca_{L}^\pm\right).$$ $\phi_{\pm}$ has this nice property that its restriction of to $\mscpn_{(d),L}$ is a fibration with fibers equal to the $\cc^*$-fixed locus of the flag manifold $$\fl\left((D),\cc^L\right)$$ under the natural \footnote{We have identified $\Ext^1_{\A}(U,\V)$ (or $\Ext^1_{\A}(\V,U)$) with $\cc^L$.} induced $\cc^*$-action. Here $(D)=D_1,D_2,\dots$ is a sequence of nonnegative integers defined by $$D_j=\sum_{i\ge j}d_i.$$
We will denote the flag manifold above by $\f$ when $(D)$ and $L$ are clear from the context.
Conversely, given $(d)$, a $\cc^*$-invariant $\V\in \msca$, and a $\cc^*$-fixed point of the corresponding flag manifold, one can find a unique the $\cc^*$-invariant $\V' \in \mscpn$ fitting into either of the short exact sequences above. The proof of all these claims can be found in \cite{Nakajima-Nagao, Nagao-toric}.

The following corollary is resulted immediately from the discussion above and Proposition \ref{prop:ext}, ii):
\begin{cor}\label{cor:relation-up-low}
$$\msca_{L}^+=\msca_{L+\dim U_{\rho_0}}^-.$$
\end{cor}

We are now ready to express and prove the wall-crossing formula:

\begin{thm}\label{thm:wall-crossing}
$$Z_{\zeta_\alpha^{-}}(\mathbf{q})=\left(1-(-1)^{\dim U_{\rho_0}}\mathbf{q}^{\dim U}\right)^{-\dim{U_{\rho_0}}}
Z_{\zeta_\alpha^{+}}(\mathbf{q}).$$
\end{thm}
\begin{proof}
By Corollary \ref{cor:euler char of moduli space} and using the property of the fibrations defined above (the restrictions of $\phi_\pm$) we can write
\begin{align*}
Z_{\zeta_\alpha^{\pm}}(\mathbf{q})&=\sum_{\mathbf{v}'}(-1)^{v'_{\rho_0}}\chi(\mscpn)\mathbf{q}^{\mathbf{v}'} \\
&=\sum_{\mathbf{v}', (d), L}(-1)^{v'_{\rho_0}}\chi(\mscpn_{(d),L})\mathbf{q}^{\mathbf{v}'} \\
&=\sum_{\mathbf{v}, (d), L}(-1)^{v_{\rho_0}+(\sum md_{m})\dim U_{\rho_0}}\chi(\f)\cdot\chi(\msca_{L}^\pm)\mathbf{q}^{\mathbf{v}+(\sum md_{m})\dim U} \\
&=\sum_{\mathbf{v},(d), L}(-1)^{(\sum md_{m})\dim U_{\rho_0}}\chi(\f)\mathbf{q}^{(\sum md_{m})\dim U}\cdot(-1)^{v_{\rho_0}}\chi(\msca_{L}^\pm)\mathbf{q}^{\mathbf{v}}\\
&=\sum_{\mathbf{v},L}\left(1-(-1)^{\dim U_{\rho_0}}\mathbf{q}^{\dim(U)}\right)^{-L}
(-1)^{v_{\rho_0}}\chi(\msca_{L}^\pm)\mathbf{q}^{\mathbf{v}}.
\end{align*}
Then by Corollary \ref{cor:relation-up-low},
\begin{align*}
&Z_{\zeta_\alpha^{-}}(\mathbf{q})\\
&=\sum_{\mathbf{v},L}\left(1-(-1)^{\dim U_{\rho_0}}\mathbf{q}^{\dim U}\right)^{-L}
(-1)^{v_{\rho_0}}\chi(\msca_{L-\dim U_{\rho_0}}^+)\mathbf{q}^{\mathbf{v}}\\
&=\sum_{\mathbf{v},L}\left(1-(-1)^{\dim U_{\rho_0}}\mathbf{q})^{\dim U}\right)^{-L-\dim U_{\rho_0}}(-1)^{v_{\rho_0}}\chi(\msca_{L}^+)\mathbf{q}^{\mathbf{v}}  \\
&=\left(1-(-1)^{\dim U_{\rho_0}}\mathbf{q}^{\dim{U}}\right)^{-\dim{U_{\rho_0}}}
Z_{\zeta_\alpha^{+}}(\mathbf{q}).
\end{align*}
\end{proof}

%%%--------------------
\section{Applications of the wall crossing formula}
\subsection{DT/PT/NCDT-correspondence}
Let $$\nu_{I}: I_{n}(Y,\beta)\rightarrow \zz \text{\hspace{.5cm} and \hspace{.5cm}} \nu_{P}: P_{n}(Y,\beta)\rightarrow \zz.$$ be the Behrend functions defined on
$I_{n}(Y,\beta)$ and $P_{n}(Y,\beta)$. These two moduli spaces admit symmetric obstruction theory (see \cite{Behrend-microlocal}), and hence
both DT and PT invariants can be expressed as the Euler characteristics weighted by the Behrend function:
$$I_{n,\beta}=\chi(I_{n}(Y,\beta),\nu_{I}) \text{ and } P_{n,\beta}=\chi(P_{n}(Y,\beta),\nu_{P}).$$ By Corollary \ref{cor:PT/DT chambers} we have
\begin{align} \label{equ:Z+/-=Z_PT/DT}
Z^Y_{DT}\left(\prod_{\rho \in \Irr(G)}q^{\dim \rho}; (q^{-1}_\rho)_{\rho\in \irr(G)}\right)&=Z_{\zeta^{im,-}}(\mathbf{q}),\\ \nonumber
Z^Y_{PT}\left(\prod_{\rho \in \Irr(G)}q^{\dim \rho}; (q^{-1}_\rho)_{\rho\in \irr(G)}\right)&=Z_{\zeta^{im,+}}(\mathbf{q}).
\end{align}
Recall that $\zeta^{im}=(-r,1,\cdots,1)$.
\begin{defn} \label{defn:subset of + roots}
We define
$$\R_{-}^{+,re}=\{\alpha\in \R^{+,re}| \alpha\cdot\zeta^{im}<0\}$$ and
$$\R_{+}^{+,re}=\{\alpha\in \R^{+,re}| \alpha\cdot\zeta^{im}>0\}.$$
\end{defn}
\begin{example}\label{E7c}
In the Example \ref{E7}, $r=17$ and  $$\zeta^{im}=(-17,1,1,1,1,1,1,1).$$ Let
$\ah=m\ah^{im}+\beta$. Then
\begin{align*}
&\ah \cdot \zeta^{im}\\
&=-17m+(2m+2)+(3m+3)+(2m+2)\\
&+(4m+4)+(3m+3)+(2m+2)+(m+1)\\
&=17>0,
\end{align*}
and hence $$\ah=m\ah^{im}+\beta\in \R_{+}^{+,re}.$$

If we let
$$ \beta'=
\begin{array}{ccccccc}
&&&-2&&& \\
-2&-3&-4&-3&-2&-1,
\end{array}
$$
which is minus a positive root for $E_7$. Then
$$\ah'=m\ah^{im}+{\beta}'=
\begin{array}{ccccccc}
&&&2m-2&&& \\
m&2m-2&3m-3&4m-4&3m-3&2m-2&m-1
\end{array}
$$

is a positive real root of $\widehat{E}_7$ and
\begin{align*}
&\ah'\cdot \zeta^{im}\\
&=-17m+(2m-2)+(3m-3)+(2m-2)\\&+(4m-4)+(3m-3)+(2m-2)+(m-1)\\
&=-17<0,
\end{align*}
and hence $\ah=m\ah^{im}+{\beta}\in R_{-}^{+,re}$.

\end{example}

The wall crossing formula in Theorem \ref{thm:wall-crossing} gives the following relations among DT/PT/NCDT invariants:

\begin{thm}\label{thm:NCDT-DT}
\begin{align*}
Z_{NCDT}^{\Q}(\mathbf{q})=&\prod_{\alpha\in \R_{+}^{+,re}}\left(1-(-q_{\rho_0})^{\alpha_{\rho_0}}\prod_{\rho\in \irr(G)}q_\rho^{\alpha_\rho}\right)^{-\alpha_{\rho_0}}\\&\cdot Z^Y_{DT}\left(\prod_{\rho \in \Irr(G)}q^{\dim \rho}; (q^{-1}_\rho)_{\rho\in \irr(G)}\right),\end{align*} and
$$Z^Y_{PT}(\mathbf{q})=\prod_{\alpha\in \R_{-}^{+,re}}\left(1-(-q_{\rho_0})^{\alpha_{\rho_0}}\prod_{\rho\in \irr(G)}q_\rho^{\alpha_\rho}\right)^{-\alpha_{\rho_0}}.$$
\end{thm}
\begin{proof}
If $\zeta$ is a stability parameter such that each
$\zeta_\rho>0$ for all $\rho\in \Irr(G)$ then $\MM_{\zeta}^{ss}(\mathbf{v})=\emptyset$  except when $\mathbf{v}=0$.
On the other hand, if $\zeta$ is such that that $\zeta_\rho<0$ for all $\rho\in \Irr(G)$ then $\MM_{\zeta}^{s}(\mathbf{v})$  is the same as the moduli space of cyclic representations of $\Q$ defined by Szendroi in \cite{Szendroi}. Now the theorem follows immediately from (\ref{equ:Z+/-=Z_PT/DT}) and Theorem \ref{thm:wall-crossing}.
\end{proof}

%%%---------------------------------------
\subsection{DT crepant resolution conjecture and the proof of Theorem \ref{thm:DTorbifold}}

For the orbifold $X=\cc^3/G$, Bryan and Young \cite{Bryan-Young} formulated a Donaldson-Thomas crepant resolution conjecture. We briefly review their formulation here.

Let $\rep$ be the set of all representations of $G$ up to isomorphism. For any $\eta\in \rep$ let $\mbox{Hilb}^{\eta}(X)$ be the Hilbert scheme of $G$-equivariant zero dimensional subschemes $Z \subset \cc^3$ such that $H^0(Z)\cong \eta$ as a $G$-representation. The orbifold Donaldson-Thomas
partition function of $X$ is defined by
\begin{equation}
Z_{DT}^{X}(\mathbf{q})=\sum_{\eta \in \rep} I_{\eta}\mathbf{q}^{\eta},
\end{equation}
where $I_\eta=\chi(\mbox{Hilb}^{\eta}(X),\nu)$ is the weighted Euler characteristic in the sense of
Behrend \cite{Behrend-microlocal}, and $\nu$ is the Behrend function on $\mbox{Hilb}^{\eta}(X)$.

The crepant resolution conjecture for the orbifold Donaldson-Thomas invariants is stated as follows (see \cite{Bryan-Young}):
\begin{equation}\label{crepant}
Z_{DT}^{X}\left((q_\rho)_{\rho\in \Irr(G)}\right)=M(-q)^{-\chi(Y)}Z_{DT}^Y\left(q;(v_\rho)_{\rho \in \irr(G)}\right)Z^Y_{DT}\left(q; (v_\rho^{-1})_{\rho\in \irr(G)}\right),
\end{equation}
under the following change of variables
$$
\begin{cases}
v_\rho=q_\rho, & \rho\in \irr(G); \\
q=\prod_{\rho \in \Irr(G)}q_\rho^{\dim \rho},
\end{cases}
$$
where $M(q)$ is the MacMahon function.

Assuming the DT/PT-correspondence (see Conjecture \ref{con:DY/PT/GW}), we get
\begin{equation} \label{equ:DT/PT}
Z^Y_{PT}\left(\prod_{\rho\in \Irr(G)}q_\rho^{\dim \rho};(q_\rho)_{\rho \in \irr(G)}\right)=M(-q)^{-\chi(Y)}Z^Y_{DT}\left(q;(v_\rho)_{\rho\in \irr(G)}\right).
\end{equation}

By Proposition \ref{relation-roots}, for any
$\alpha\in \R_{\pm}^{+,re}$, there exists a positive root $\beta\in R^{+}$ such that
$\alpha=m\ah^{im} \pm \beta$.
So
\begin{align*}
&\prod_{\alpha\in \R_{\pm}^{+,re}}\left(1-(-q_{\rho_0})^{\alpha_{\rho_0}}\prod_{\rho\in \irr(G)}q_\rho^{\alpha_\rho}\right)^{-\alpha_{\rho_0}} \\
&=
\prod_{\beta\in R^{+}}\prod_{m=1}^{\infty}\left(1-\left(-\prod_{\rho\in \Irr(G)}q_\rho^{\dim \rho}\right)^{m}
 \prod_{\rho\in \irr(G)}q_\rho^{\pm \beta_\rho}\right)^{-m}.
\end{align*}
Thus,
\begin{align} \label{equ:formula for PT}
&Z^Y_{PT}\left(\prod_{\rho\in \Irr(G)}q_\rho^{\dim \rho};(q_\rho^{\pm 1})_{\rho \in \irr(G)}\right)\\\nonumber
&=
\prod_{\beta\in R^{+}}\prod_{m=1}^{\infty}\left(1-\left(-\prod_{\rho\in \Irr(G)}q_\rho^{\dim \rho}\right)^{m}
 \prod_{\rho\in \irr(G)}q_\rho^{\pm\beta_\rho}\right)^{-m}.
\end{align}
On the other hand, from the discussion in \cite[Section 7]{Joyce-Song}, the orbifold Donaldson-Thomas invariants of the orbifold $X$ are the same as the Szendroi invariants for the ADE McKay quivers. In our case this translates into $$Z_{X}^{DT}(\mathbf{q})=Z_{NCDT}^{\Q}(\mathbf{q}),$$
and hence
\begin{align*}
Z^{X}_{DT}(\mathbf{q})&=
M(-q)^{-\chi(Y)}Z^Y_{DT}\left(\prod_{\rho\in \Irr(G)}q_\rho^{\dim \rho};(q_\rho)_{\rho \in \irr(G)}\right)\\&\cdot Z^Y_{DT}\left(\prod_{\rho\in \Irr(G)}q_\rho^{\dim \rho};(q^{-1}_\rho)_{\rho \in \irr(G)}\right)\\
&=M(-q)^{-\chi(Y)}Z^Y_{DT}\left(q;(v_\rho)_{\rho\in \irr(G)}\right)\cdot Z^Y_{DT}\left(q;(v^{-1}_\rho)_{\rho\in \irr(G)}\right),
\end{align*}
which is exactly the the DT Crepant Resolution formula (\ref{crepant}).

Now by (\ref{equ:DT/PT}) and (\ref{equ:formula for PT}) we get

\begin{align} \label{equ:formula for DT}
&Z^Y_{DT}\left(\prod_{\rho\in \Irr(G)}q_\rho^{\dim \rho};(q_\rho^{\pm 1})_{\rho \in \irr(G)}\right)\\\nonumber
&=M(-q)^{\chi(Y)}
\prod_{\beta\in R^{+}}\prod_{m=1}^{\infty}\left(1-\left(-\prod_{\rho\in \Irr(G)}q_\rho^{\dim \rho}\right)^{m}
 \prod_{\rho\in \irr(G)}q_\rho^{\pm\beta_\rho}\right)^{-m}.
\end{align}

This finishes the proof of Theorem \ref{thm:DTorbifold}.

%%%%%%%%%%%%%%%%%%%%%%%%%%%%%%%%%%%%%%%%%%%%%%%%%%%
%%%%%%%
\section{GW/PT correspondence} \label{sec:GW}

\subsection{The GW-partition function of $Y$ and the proof of Theorem \ref{thm:GW}}
In this section we compute the Gromov-Witten invariants of $Y=S \times \cc$ using deformation and localization techniques. Recall that $S$ is the minimal resolution of $\cc^2/G$ corresponding to one of the ADE Dynkin diagrams. $S$ is equipped with a natural $\cc^*$-action. The weights of this action on the fibers of canonical bundle $K_S$ at the fixed points of $S$ is equal to -2. We introduce a $\cc^*$ with weight -2 on the trivial direction of $Y$. This way $Y$ is equipped with a $\cc^*$-action in such a way it acts trivially on the fibers of the canonical bundle $K_Y$ at the fixed points.

Let $N_{g,\beta}$ be the genus $g$, $\cc^*$-equivariant Gromov-Witten invariant of $Y$ in the class $\beta \in H_2(Y,
\zz)$. Note that $$H_2(Y,\zz)\cong H_2(S,\beta).$$

By using virtual localization, we first express Gromov-Witten invariants of $Y$ in terms of the reduced Gromov-Witten invariants of $S$ with a single Hodge class (see \cite[Section 5]{Maulik-An}). The weight of the $\cc^*$-action on the normal bundle $N_{S|Y}$ is -2. Hence we can write

\[
N_{g,\beta}= \int _{[\M _{g,0}
(S,\beta)]^{vir}}e (-R^{\bullet }p _{*}f^{*}N_{Y/W})
\]
where $\M _{g,0} (S,\beta)$ is the moduli space of stable maps, $$p :\mathcal{C}\to
\M _{g,0} (S,\beta)$$ is the universal curve, $f:\mathcal{C}\to S$ is the universal
map, and $e$ is the equivariant Euler class.

Since the line bundle $N_{S/Y} $ is trivial up to the $\cc^{*}$
action, and $p$ is a family of genus $g$ curves, we get
\[
R^{\bullet }p_{*}f^{*}N_{S/Y} = R^{0}p_{*}f^{*}N_{S/Y}-R^{1}p_{*}f^{*}N_{S/Y}=\O\otimes \cc_{-2t}-\mathbb{E}_g^\vee\otimes \cc_{-2t}
\]
where $\cc_{-2t}$ is the $\cc^{*}$-representation of weight -2, and $\mathbb{E}_g$ is the Hodge bundle over $\M _{g,0} (S,\beta)$.
Consequently, %we have
%\[
%e (-R^{\bullet }p_{*}f^{*}N_{S/Y})= \frac{-1}{2t}(-1)^g\lambda_g
%\]
%and so
\begin{align*}
N_{g,\beta}&=\int _{[\M _{g,0} (S,\beta)]^{vir}}\frac{-1}{2t}(-1)^g\lambda_g\\
&=-\left \langle (-1)^g\lambda_g \right \rangle^{S,\text{red}}_{g,\beta}.
\end{align*}
The last expression is the reduced Gromov-Witten invariant of $S$ with a single Hodge class. Note that to obtain reduced Gromov-Witten invariants of $S$ one needs to remove a trivial factor with $\cc^*$-weight 2 from the ordinary Gromov-Witten obstruction theory of $S$ (see \cite{Maulik-An}).

In order to evaluate $$\left \langle (-1)^g\lambda_g \right \rangle^{S,\text{red}}_{g,\beta}$$ we reduce to the case that $S=\mathcal{A}_1$, i.e. $S$ is the resolution of type $A_1$. This can be achieved by using deformation invariance property of Gromov-Witten invariants. We consider the versal deformation space of the resolution $S\to \cc^2/G$ of type ADE constructed by Brieskorn. The versal deformation space is identified with the corresponding complexified root space $\res$. An affine generic line passing through $0 \in \res$ corresponds to a topologically trivial family of surfaces with affine generic fiber and with the central fiber isomorphic to $S$. By deforming this line away from the origin one gets another topologically trivial family of surface with affine generic fibers and a number of special fibers each isomorphic to $\mathcal{A}_1$. These special fibers are in bijection with the set of positive roots. The curve class of the single rational curve in each special fiber in the family is the same as the corresponding root class. Using some comparison arguments on the virtual classes of these two families and the reduced virtual classes of $S$ and the special fibers above, and applying the deformation invariance of Gromov-Witten invariants (see \cite{Maulik-An}), one finds that
$$\left \langle (-1)^g\lambda_g \right \rangle^{S,\text{red}}_{g,\beta}=\begin{cases}\left \langle (-1)^g\lambda_g \right \rangle^{\mathcal{A}_1,\text{red}}_{g,d} & \text{if } \beta \text{ is } d \text{ times a positive root class,}\\ 0 & \text{ otherwise.} \end{cases}$$
Define the series $$F_d(\lambda)=\sum_{g=0}^{\infty}\left \langle(-1)^g\lambda_g \right\rangle^{\mathcal{A}_1,\text{red}}_{g,d}\lambda^{2g-2}.$$ $F_d$ has been evaluated in \cite[Section 5]{Maulik-An}: $$F_d(\lambda)=\frac{1}{d}\left(2\sin \frac{d\lambda}{2}\right)^{-2}.$$
Putting altogether, and using the identity
$$\left(2\sin \frac{d\lambda}{2}\right)^{-2}=\frac{-\operatorname{e}^{id\lambda}}{(1-\operatorname{e}^{id\lambda})^2}$$
we get the evaluation for the partition function given in Theorem \ref{thm:GW}.

\subsection{GW/PT-correspondence and the proof of Theorem \ref{thm:PT}}
By Theorem \ref{thm:GW} we have
$$Z_{GW}^Y(\lambda;\mathbf{t})=\prod_{\beta\in R^{+}}\prod_{m=1}^{\infty}
(1-\mathbf{t}^{\beta}(e^{i\lambda})^{m})^{-m}.$$
Replacing $e^{i\lambda}=-q$ yields
$$Z^Y_{GW}(q; \mathbf{t})=\prod_{\beta \in R^+}\prod_{m=1}^{\infty}
(1-\mathbf{t}^{\beta}(-q)^{m})^{-m}.$$ By Proposition \ref{relation-roots}, after setting
$$
\begin{cases}
t_\rho=q^{-1}_\rho, & \rho\in \irr(G); \\
q=\prod_{\rho \in \Irr(G)}q_\rho^{\dim \rho},
\end{cases}
$$ we get
\begin{align*}
&\prod_{\beta\in R^{+}}\prod_{m=1}^{\infty}
(1-\mathbf{t}^{\beta}(-q)^{m})^{-m} \\
&=\prod_{{\beta}\in R^{+}}\prod_{m=1}^{\infty}
(1-\mathbf{t}^{\beta}((-q_{\rho_0})\prod_{\rho \in \irr(G)}q_\rho^{\dim\rho})^{m})^{-m} \\
&=\prod_{{\beta}\in R^{+}}\prod_{m=1}^{\infty}
(1-(-q_{\rho_0})^{m}\prod_{\rho \in \irr(G)} q_\rho^{m \dim \rho-d_\rho})^{-m} \\
&=\prod_{\alpha\in \R_{-}^{+,re}}(1-(-q_{\rho_0})^{\alpha_{\rho_0}}\prod_{\rho \in \irr(G)}q_\rho^{\alpha_\rho})^{-\alpha_{\rho_0}},
\end{align*}
which is exactly the expression in Theorem \ref{thm:NCDT-DT}. This establishes the GW/PT correspondence and finishes the proof of Theorem \ref{thm:PT}.

%%%%---------------------------
\section{An example} \label{sec:example}
In this section we consider the case $G=\mathbb{D}_{12}$, the binary dihedral group in 12 elements, corresponding to the root system of type $D_5$. Let $$\{\rho_0,\rho_1,\dots\rho_5\}$$ be the set of irreducible representations of $\mathbb{D}_{12}$. As  usual $\rho_0$ stands for the trivial representation, $\rho_1, \rho_2,$ and $\rho_5$ are 1-dimensional, and $\rho_3$ and $\rho_4$ are 2-dimensional. The set of all edges in the corresponding quiver $\Q$ is as follows (see Figure \ref{fig:ADE}):
$$\left\{r_{\rho_0\rho_3},r_{\rho_3\rho_0},r_{\rho_1\rho_4},r_{\rho_4\rho_1}r_{\rho_2\rho_3},r_{\rho_3\rho_2},r_{\rho_3\rho_4},r_{\rho_4\rho_3}
,r_{\rho_4\rho_5},r_{\rho_5\rho_4}\right\}.$$ We denote the variable corresponding to $\rho_i$ by $q_i$, and as before $$\mathbf{q}^{\mathbf{v}}=q_0^{v_0}q_1^{v_1}q_2^{v_2}q_3^{v_3}q_4^{v_4}q_5^{v_5}$$ for any $\mathbf{v}=(v_i) \in \zz_{\ge 0}^6$.

The subset $\R^{+,re}_{-}$ of the set of real positive roots for the extended Dynkin diagram $\widehat{D}_5$ (See Definition \ref{defn:subset of + roots}) is given by ($m=1,2,3,\dots$):
\begin{center}
\small{
\begin{tabular}{@{}l@{}l@{}l@{}l@{}l@{}}
 \mbox{\renewcommand{\arraystretch}{.5} \bigg\{\begin{tabular} {l@{}l@{}l@{}l@{}}
     &~m &~m-1&\\
   m&~2m&~2m&~m,
   \end{tabular}}&
\mbox{\renewcommand{\arraystretch}{.5}
\begin{tabular}{l@{}l@{}l@{}l@{}}
     &~m &~m&\\
    m-1&~2m&~2m&~m,
   \end{tabular}}&
\mbox{\renewcommand{\arraystretch}{.5}
\begin{tabular}{l@{}l@{}l@{}l@{}}
     &~m&~m&\\
    m&~2m-1&~2m&~m,
   \end{tabular}}&
\mbox{\renewcommand{\arraystretch}{.5}
\begin{tabular}{l@{}l@{}l@{}l@{}}
    &~m&~m&\\
    m&~2m&~2m-1&~m,
   \end{tabular}}
   \quad \\ \\
\mbox{\renewcommand{\arraystretch}{.5}
\begin{tabular}{l@{}l@{}l@{}l@{}}
    &~m&~m&\\
    m&~2m&~2m&~m-1,
   \end{tabular}}&
  \mbox{\renewcommand{\arraystretch}{.5} \begin{tabular}{l@{}l@{}l@{}l@{}}
     &~m&~m&\\
    m-1&~2m-1&~2m&~m,   \end{tabular}}&
\mbox{\renewcommand{\arraystretch}{.5}
\begin{tabular}{l@{}l@{}l@{}l@{}}
      &~m&~m&\\
    m&~2m-1&~2m-1&~m,
   \end{tabular}}&
\mbox{\renewcommand{\arraystretch}{.5}
\begin{tabular}{l@{}l@{}l@{}l@{}}
      &~m&~m&\\
    m&~2m&~2m-1&~m-1,
   \end{tabular}}
   \quad \\ \\
\mbox{\renewcommand{\arraystretch}{.5}
\begin{tabular}{l@{}l@{}l@{}l@{}}
    &~m&~m-1&\\
    m&~2m&~2m-1&~m,
   \end{tabular}}&
\mbox{\renewcommand{\arraystretch}{.5}
\begin{tabular}{l@{}l@{}l@{}l@{}}
    &~m&~m&\\
    m-1&~2m-1&~2m-1&~m,
   \end{tabular}}&
\mbox{\renewcommand{\arraystretch}{.5}
\begin{tabular}{l@{}l@{}l@{}l@{}}
    &~m&~m&\\
    m&~2m-1&~2m-1&~m-1,
   \end{tabular}}&
\mbox{\renewcommand{\arraystretch}{.5}
\begin{tabular}{l@{}l@{}l@{}l@{}}
     &~m&~m-1&\\
    m&~2m&~2m-1&~m-1,
   \end{tabular}}
   \quad \\ \\
\mbox{\renewcommand{\arraystretch}{.5}
\begin{tabular}{l@{}l@{}l@{}l@{}}
      &~m&~m-1&\\
    m&~2m-1&~2m-1&~m,
   \end{tabular}}&
\mbox{\renewcommand{\arraystretch}{.5}
\begin{tabular}{l@{}l@{}l@{}l@{}}
      &~m&~m&\\
    m-1&~2m-1&~2m-1&~m-1,
   \end{tabular}}&
\mbox{\renewcommand{\arraystretch}{.5}
\begin{tabular}{l@{}l@{}l@{}l@{}}
       &~m&~m-1&\\
    m&~2m-1&~2m-1&~m-1,
   \end{tabular}}&
    \mbox{\renewcommand{\arraystretch}{.5} \begin{tabular}{l@{}l@{}l@{}l@{}}
      &~m&~m-1&\\
    m-1&~2m-1&~2m-1&~m,
   \end{tabular}}
   \quad \\ \\
\mbox{\renewcommand{\arraystretch}{.5}
\begin{tabular}{l@{}l@{}l@{}l@{}}
       &~m&~m-1&\\
    m-1&~2m-1&~2m-1&~m-1,
   \end{tabular}}&
\mbox{\renewcommand{\arraystretch}{.5}
\begin{tabular}{l@{}l@{}l@{}l@{}}
       &~m&~m-1&\\
    m-1&~2m-1&~2m-2&~m-1,
   \end{tabular}}&
\mbox{\renewcommand{\arraystretch}{.5}
\begin{tabular}{l@{}l@{}l@{}l@{}}
       &~m&~m-1&\\
    m&~2m-1&~2m-2&~m-1,
   \end{tabular}}&
\mbox{\renewcommand{\arraystretch}{.5}
\begin{tabular}{l@{}l@{}l@{}l@{}}
       &~m&~m-1&\\
    m-1&~2m-2&~2m-2&~m-1
   \end{tabular}}\;\bigg\}.\\
\end{tabular}
}
\end{center}

Let $\zeta^{im, +}\in \rr^6$ be the corresponding PT stability parameter (see Section \ref{sec:chambers DT-PT}).
Then by the repeated use of the wall crossing formula (Theorem \ref{thm:wall-crossing}) we obtain the generating function:
\begin{align*}
&Z_{\zeta^{im,+}}(\mathbf{q})= \prod_{m=1}^{\infty}\\
&\left(1-(-q_0)^{m}q_1^{m-1}q_2^{m}q_3^{2m}q_4^{2m}q_5^{m}\right)^{-m}\left(1-(-q_0)^{m}q_1^{m}q_2^{m-1}q_3^{2m}q_4^{2m}q_5^{m}\right)^{-m} \\
&\left(1-(-q_0)^{m}q_1^{m}q_2^{m}q_3^{2m-1}q_4^{2m}q_5^{m}\right)^{-m}\left(1-(-q_0)^{m}q_1^{m}q_2^{m}q_3^{2m}q_4^{2m-1}q_5^{m}\right)^{-m}\\
&\left(1-(-q_0)^{m}q_1^{m}q_2^{m}q_3^{2m}q_4^{2m}q_5^{m-1}\right)^{-m} \left(1-(-q_0)^{m}q_1^{m}q_2^{m-1}q_3^{2m-1}q_4^{2m}q_5^{m}\right)^{-m}\\
&\left(1-(-q_0)^{m}q_1^{m}q_2^{m}q_3^{2m-1}q_4^{2m-1}q_5^{m}\right)^{-m}\left(1-(-q_0)^{m}q_1^{m}q_2^{m}q_3^{2m}q_4^{2m-1}q_5^{m-1}\right)^{-m}\\
&\left(1-(-q_0)^{m}q_1^{m-1}q_2^{m}q_3^{2m}q_4^{2m-1}q_5^{m}\right)^{-m}\left(1-(-q_0)^{m}q_1^{m}q_2^{m-1}q_3^{2m-1}q_4^{2m-1}q_5^{m}\right)^{-m}\\
&\left(1-(-q_0)^{m}q_1^{m}q_2^{m}q_3^{2m-1}q_4^{2m-1}q_5^{m-1}\right)^{-m} \left(1-(-q_0)^{m}q_1^{m-1}q_2^{m}q_3^{2m}q_4^{2m-1}q_5^{m-1}\right)^{-m}\\
&\left(1-(-q_0)^{m}q_1^{m-1}q_2^{m}q_3^{2m-1}q_4^{2m-1}q_5^{m}\right)^{-m} \left(1-(-q_0)^{m}q_1^{m}q_2^{m-1}q_3^{2m-1}q_4^{2m-1}q_5^{m-1}\right)^{-m} \\
&\left(1-(-q_0)^{m}q_1^{m-1}q_2^{m}q_3^{2m-1}q_4^{2m-1}q_5^{m-1}\right)^{-m} \left(1-(-q_0)^{m}q_1^{m-1}q_2^{m-1}q_3^{2m-1}q_4^{2m-1}q_5^{m}\right)^{-m}\\
&\left(1-(-q_0)^{m}q_1^{m-1}q_2^{m-1}q_3^{2m-1}q_4^{2m-1}q_5^{m-1}\right)^{-m}\left(1-(-q_0)^{m}q_1^{m-1}q_2^{m}q_3^{2m-1}q_4^{2m-2}q_5^{m-1}\right)^{-m} \\
&\left(1-(-q_0)^{m}q_1^{m-1}q_2^{m-1}q_3^{2m-1}q_4^{2m-2}q_5^{m-1}\right)^{-m}\left(1-(-q_0)^{m}q_1^{m-1}q_2^{m-1}q_3^{2m-2}q_4^{2m-2}q_5^{m-1}\right)^{-m}.
\end{align*}

The PT partition function of $Y$ in this case is then obtained by the following change of variables:
$q=q_0q_1q_2q^2_3q^2_4q_5$ and $t_i=q_i^{-1}$, for $i=1,2,\dots,5$:

\begin{align*}
&Z_{PT}^Y(q,\mathbf{t})=\prod_{m=1}^{\infty} \\
&\left(1-t_1(-q)^{m}\right)^{-m}\left(1-t_2(-q)^{m}\right)^{-m}\left(1-t_3(-q)^{m}\right)^{-m}\left(1-t_4(-q)^{m}\right)^{-m} \\
&\left(1-t_5(-q)^{m}\right)^{-m}\left(1-t_2t_3(-q)^{m}\right)^{-m}\left(1-t_3t_4(-q)^{m}\right)^{-m}\left(1-t_4t_5(-q)^{m}\right)^{-m} \\
&\left(1-t_1t_4(-q)^{m}\right)^{-m} \left(1-t_2t_3t_4(-q)^{m}\right)^{-m}\left(1-t_3t_4t_5(-q)^{m}\right)^{-m}\\&
\left(1-t_1t_4t_5(-q)^{m}\right)^{-m}\left(1-t_1t_3t_4q^{m}\right)^{-m}\left(1-t_2t_3t_4t_5(-q)^{m}\right)^{-m}\\&
\left(1-t_1t_3t_4t_5(-q)^{m}\right)^{-m}\left(1-t_1t_2t_3t_4(-q)^{m}\right)^{-m} \left(1-t_1t_2t_3t_4t_5(-q)^{m}\right)^{-m}\\&\left(1-t_1t_3t^2_4t_5(-q)^{m}\right)^{-m}\left(1-t_1t_2t_3t^2_4t_5(-q)^{m}\right)^{-m}
\left(1-t_1t_2t_3^2t^2_4t_5(-q)^{m}\right)^{-m}.
\end{align*}

The set of all positive roots of $D_5$ root system is as follows
\begin{center}
\small{
\begin{tabular}{@{}l@{}l@{}l@{}l@{}l@{}l@{}l@{}l@{}l@{}l@{}}
\mbox{\renewcommand{\arraystretch}{.5}
\bigg\{\begin{tabular}{l@{}l@{}l@{}l@{}}
     & &1&\\
    0&0&0&0, \;
   \end{tabular}}&
\mbox{\renewcommand{\arraystretch}{.5}
\begin{tabular}{l@{}l@{}l@{}l@{}}
     & &0&\\
    1&0&0&0, \;
   \end{tabular}}&
\mbox{\renewcommand{\arraystretch}{.5}
\begin{tabular}{l@{}l@{}l@{}l@{}}
     & &0&\\
    0&1&0&0, \;
   \end{tabular}}&
\mbox{\renewcommand{\arraystretch}{.5}
\begin{tabular}{l@{}l@{}l@{}l@{}}
     & &0&\\
    0&0&1&0, \;
   \end{tabular}}&
\mbox{\renewcommand{\arraystretch}{.5} \begin{tabular}{l@{}l@{}l@{}l@{}}
     & &0&\\
   0&0&0&1,
   \end{tabular}}&
\mbox{\renewcommand{\arraystretch}{.5} \begin{tabular}{l@{}l@{}l@{}l@{}}
     & &0&\\
    1&1&0&0, \;
   \end{tabular}}&
\mbox{\renewcommand{\arraystretch}{.5}
\begin{tabular}{l@{}l@{}l@{}l@{}}
     & &0&\\
    0&1&1&0, \;
   \end{tabular}}&
\mbox{\renewcommand{\arraystretch}{.5}
\begin{tabular}{l@{}l@{}l@{}l@{}}
     & &0&\\
    0&0&1&1, \;
   \end{tabular}}&
\mbox{\renewcommand{\arraystretch}{.5}
\begin{tabular}{l@{}l@{}l@{}l@{}}
     & &1&\\
    0&0&1&0, \;
   \end{tabular}}&
\mbox{\renewcommand{\arraystretch}{.5}
\begin{tabular}{l@{}l@{}l@{}l@{}}
     & &0&\\
    1&1&1&0, \;
   \end{tabular}}\\
\quad \\
\mbox{\renewcommand{\arraystretch}{.5}
\begin{tabular}{l@{}l@{}l@{}l@{}}
     & &0&\\
   0&1&1&1,
   \end{tabular}}&
\mbox{\renewcommand{\arraystretch}{.5}
\begin{tabular}{l@{}l@{}l@{}l@{}}
     & &1&\\
    0&0&1&1,
   \end{tabular}}&
\mbox{\renewcommand{\arraystretch}{.5}
\begin{tabular}{l@{}l@{}l@{}l@{}}
     & &1&\\
    0&1&1&0,
   \end{tabular}}&
\mbox{\renewcommand{\arraystretch}{.5}
\begin{tabular}{l@{}l@{}l@{}l@{}}
     & &0&\\
    1&1&1&1,
   \end{tabular}}&
\mbox{\renewcommand{\arraystretch}{.5}
\begin{tabular}{l@{}l@{}l@{}l@{}}
     & &1&\\
    0&1&1&1,
   \end{tabular}}&
\mbox{\renewcommand{\arraystretch}{.5}
\begin{tabular}{l@{}l@{}l@{}l@{}}
     & &1&\\
    1&1&1&0,
   \end{tabular}}&
\mbox{\renewcommand{\arraystretch}{.5}
\begin{tabular}{l@{}l@{}l@{}l@{}}
     & &1&\\
    1&1&1&1,
   \end{tabular}}&
\mbox{\renewcommand{\arraystretch}{.5}
\begin{tabular}{l@{}l@{}l@{}l@{}}
     & &1&\\
    0&1&2&1,
   \end{tabular}}&
\mbox{\renewcommand{\arraystretch}{.5}
\begin{tabular}{l@{}l@{}l@{}l@{}}
     & &1&\\
    1&1&2&1,
   \end{tabular}}&
\mbox{\renewcommand{\arraystretch}{.5}
\begin{tabular}{l@{}l@{}l@{}l@{}}
     & &1&\\
    1&2&2&1
   \end{tabular}}\bigg\}.\\
\end{tabular}
}
\end{center}
One can then see easily the correspondence of the factors in $Z_{PT}(q,\mathbf{t})$ to the positive roots above in agreement with the expression in Theorem \ref{thm:PT}. The GW partition function is obtained by replacing $-q=e^{i\lambda}$ in $Z_{PT}(q,\mathbf{t})$, in agreement with Theorem \ref{thm:GW}.

\section*{Acknowledgement}
The second author would like to thank Tsinghua University, Beijing for hospitality during a visit 
in summer 2009, and Professor Jian Zhou for the valuable discussions.
 

%%%-----------------------------------------------------------------------

% ------------------------------------------------------------------------
\bibliography{mainbiblio}
\bibliographystyle{plain}
\end{document}